\documentclass[11pt]{article}
\usepackage[a4paper, top=2.5cm, left=3cm, right=3cm, bottom=2.9cm]{geometry}
\usepackage{amsmath, amssymb}
\usepackage{enumerate}

\usepackage{mathtools}
\usepackage[usenames,dvipsnames]{xcolor}
\definecolor{myred}{RGB}{255, 61, 65}
\usepackage{amsthm}
\usepackage{tikz}
\usepackage{asymptote}
\usepackage[plainpages=false,pdfpagelabels,colorlinks=true,citecolor=blue,hypertexnames=false]{hyperref}
\usepackage[labelfont={sc}]{caption}
\usepackage{tikz}
\usetikzlibrary{cd}
\usepackage[capitalize]{cleveref}
\usepackage{colonequals}
\newcommand*{\longeq}{\ratio\Longleftrightarrow}
\usepackage{multicol}
\usepackage{tikz}
\usetikzlibrary{shapes.geometric, arrows}
\tikzstyle{startstop} = [rectangle, rounded corners, minimum width=3cm, minimum height=1cm,text centered, draw=black, fill=brown!30]
\tikzstyle{decision} = [diamond, aspect=1.5, inner xsep=0pt, text centered, draw=black, fill=lime!30, text width=1.8cm]
\tikzstyle{processyes} = [rectangle, minimum width=3cm, minimum height=1cm, text centered, draw=black, fill=green!30]
\tikzstyle{processno} = [rectangle, minimum width=3cm, minimum height=1cm, text centered, draw=black, fill=red!30]
\tikzstyle{arrow} = [thick,->,>=stealth]

\newcommand{\BeHe}[6]{
\resizebox{\textwidth}{!}{\begin{tikzpicture}[scale=1/#1*15,cap=round,>=latex]
\pgfmathparse{0};
\edef\shiftvar{\pgfmathresult};

\foreach \n in {1,...,#1}{

    \pgfmathparse{gcd(\n,#1)};
    
    \ifthenelse{1=\pgfmathresult}{
    \draw[thick,xshift=3*\shiftvar cm] (0cm,0cm) circle(1cm);
    \filldraw[myred,xshift=3*\shiftvar cm] (\n*#2*360:1cm) circle(3pt);
    \filldraw[myred,xshift=3*\shiftvar cm] (\n*#3*360:1cm) circle(3pt);
    \filldraw[myred,xshift=3*\shiftvar cm] (\n*#4*360:1cm) circle(3pt);
    \filldraw[blue,xshift=3*\shiftvar cm] (\n*#5*360:1cm) circle(3pt);
    \filldraw[blue,xshift=3*\shiftvar cm] (\n*#6*360:1cm) circle(3pt);
    \filldraw[blue,xshift=3*\shiftvar cm] (\n*1*360:1cm) circle(3pt);
    \pgfmathparse{\shiftvar+1} \xdef\shiftvar{\pgfmathresult}
    }{};
}
\end{tikzpicture}
}
}

\usepackage{array}

\usepackage{todonotes}

\renewcommand{\phi}{\varphi}
\renewcommand{\epsilon}{\varepsilon}
\newcommand{\N}{\mathbb{N}}
\newcommand{\R}{\mathbb{R}}
\newcommand{\C}{\mathbb{C}}
\newcommand{\Z}{\mathbb{Z}}
\newcommand{\Q}{\mathbb{Q}}
\newcommand{\F}{\mathcal{F}}

\newcommand{\ps}[1]{[\![#1]\!]}
\newcommand{\dpc}[1]{\langle #1 \rangle}
\newcommand{\cc}{\mathrm{c}}
\newcommand{\rr}{\mathrm{r}}

\theoremstyle{definition}
\newtheorem{defi}{Definition}[section]
\newtheorem{ex}[defi]{Example}

\theoremstyle{plain}
\newtheorem{thm}[defi]{Theorem}
\newtheorem{lem}[defi]{Lemma}
\newtheorem{cor}[defi]{Corollary}
\newtheorem{prop}[defi]{Proposition}

\theoremstyle{remark}
\newtheorem{rem}[defi]{Remark}

\newcommand*\pFqskip{8mu}
\catcode`,\active
\newcommand*\pFq{\begingroup
        \catcode`\,\active
        \def ,{\mskip\pFqskip\relax}%
        \dopFq
}
\catcode`\,12
\def\dopFq#1#2#3#4#5{%
        {}_{#1}F_{#2}\biggl[\genfrac..{0pt}{}{#3}{#4};#5\biggr]%
        \endgroup
}

\catcode`,\active
\newcommand*\bigF{\begingroup
        \catcode`\,\active
        \def ,{\mskip\pFqskip\relax}%
        \dobigF
}
\catcode`\,12
\def\dobigF#1#2#3{%
        {\F}\biggl[\genfrac..{0pt}{}{#1}{#2};#3\biggr]%
        \endgroup
}

\begin{document}
\title{Algebraicity of hypergeometric functions\\ with arbitrary parameters}
\author{Florian Fürnsinn, Sergey Yurkevich}

\maketitle
\begin{abstract}
    We provide a complete classification of the algebraicity of (generalized) hypergeometric functions with no restriction on the set of their parameters. Our characterization relies on the interlacing criteria of Christol (1987) and Beukers-Heckman~(1989) for globally bounded and algebraic hypergeometric functions, however in a more general setting which allows arbitrary complex parameters with possibly integral differences. We also showcase the adapted criterion on a variety of different examples.
    \let\thefootnote\relax\footnote{
    \textit{MSC2020:} \href{https://mathscinet.ams.org/mathscinet/msc/msc2020.html?t=33C20}{33C20}, \href{https://mathscinet.ams.org/mathscinet/msc/msc2020.html?t=34M15}{34M15} (primary), \href{https://mathscinet.ams.org/mathscinet/msc/msc2020.html?t=05A15}{05A15} (secondary) }
\addtocounter{footnote}{-1}\let\thefootnote\svthefootnote
\end{abstract}

\section{Introduction} \label{sec:intro}

The \emph{hypergeometric function} with parameters $a_1,\dots,a_p,\allowbreak b_1,\dots,b_q \in \C$ for some $p,q\in \N$ and $b_k \not \in -\N$ for $k=1,\ldots, q$ is defined as the power series
\begin{equation}\label{def:hyper}
\pFq{p}{q}{a_1,\ldots, a_p}{b_1,\ldots, b_q}{x}={}_pF_q([a_1,\ldots, a_p],[b_1,\ldots, b_{q}];x) \coloneqq\sum_{n=0}^\infty \frac{(a_1)_n\cdots (a_p)_n}{(b_1)_n\cdots (b_q)_n}  \frac{x^n}{ n!},
\end{equation}
where $(a)_n\coloneqq a\cdot (a+1)\cdots (a+n-1)$ denotes the rising factorial. In this text we always assume that ${}_pF_q([a_1,\ldots, a_p],[b_1,\ldots, b_{q}];x)\in \Q\ps{x}$, i.e., the power series expansion has rational coefficients. Hypergeometric functions are exactly the generating functions of first order P-recursive sequences. More precisely, a sequence of rational numbers $(u_n)_{n \geq 0}$ is called \textit{hypergeometric} if it satisfies a linear recurrence relation with polynomial coefficients of order $1$, i.e., there exist polynomials $A(t), B(t)\in \Q[t]$ with $B(-n) \neq 0$ for $n \in \N$, s.t.
\[
u_{n+1} = \frac{A(n)}{B(n)} u_n, \text{ for all } n \geq 0.
\]
The class of hypergeometric sequences includes sequences with finite support, the geometric sequence $\alpha^n$ for $\alpha\in \Q$, the Catalan numbers (\href{https://oeis.org/A000108}{OEIS A000108}), the coefficient sequence of elliptic integrals (\href{https://oeis.org/A002894}{OEIS A002894}), or the Chebyshev numbers (\href{https://oeis.org/A211417}{OEIS A211417}) to name a few prominent examples. If $(u_n)_{n \geq 0}$ is a hypergeometric sequence, write $A(t)=\alpha\prod_{j=1}^p(t+a_j)$ and $B(t)=\beta\prod_{k=1}^q(t+b_k)$ for some $a_i,b_j\in \overline{\Q}$. Then its generating function $\sum_{n\geq 0} u_n x^n$ is given by 
\begin{equation*}
u_0\cdot\pFq{p}{q}{a_1,\ldots, a_p, 1}{b_1,\ldots, b_q}{y}, \text{ where } y = \frac{\alpha}{\beta} x.
\end{equation*}
 Conversely, the coefficient sequence of a hypergeometric function clearly forms a hypergeometric sequence.

A power series $f(x)\in \Q\ps{x}$ is called {\it algebraic} (over $\Q(x)$) if there exists a bivariate polynomial $P(x,y) \in \Q[x,y] \setminus \{ 0 \}$, such that $P(x,f(x))=0$. In this paper we give the complete answer to the following question (see \cref{thm:tree} and Figure~\ref{fig:tree}):

\begin{quote}
    \emph{Is a given hypergeometric function $f(x)\in \Q \ps{x}$ algebraic over $\Q(x)$?}
\end{quote}

This is not only an interesting question with a surprisingly beautiful answer, but also of significance throughout different fields of mathematics. Its answer is usually attributed to Beukers and Heckman~\cite[Thm.~4.8]{BH89} who solved the problem essentially. They provide an elegant algorithm, the so-called interlacing criterion (see Theorem~\ref{thm:BeHe}), which applies in the case when $p=q+1$ and $a_j,b_k \in \Q$ with $a_j, a_j-b_k\not \in \Z$ for all $j=1,\dots,p$ and $k=1,\dots,p-1$. A few years before the work by Beukers and Heckman, Christol proved in~\cite[\S VI]{Christol86} that the same criterion is necessary and that it is also sufficient under the assumption of Grothendieck's $p$-curvature conjecture. In fact, Christol's refined criterion does not need the assumption $a_j, a_j-b_k\not \in \Z$. We provide more historical details in \cref{sec:history} where we also emphasize on the contributions of Schwarz, Landau and Errera who answered the question completely for ${}_{2}F_1$'s. 


Both of the assumptions in~\cite{BH89} cannot be made without loss of generality: there exist algebraic hypergeometric series in $\Q\ps{x}$ with integer differences between top and bottom parameters or irrational parameters. For example,
\begin{equation} \label{eq:hyperintro}
\pFq{3}{2}{1/2,,\sqrt{2}+1,,-\sqrt{2}+1}{\sqrt{2},,-\sqrt{2}}{4x} = \frac{(7x - 1)(2x - 1)}{(1-4x)^{5/2}} = 1+x-6x^2 + \cdots\in \Z\ps{x},
\end{equation}
is clearly algebraic (see \cref{ex:first} for more details). This series is the generating function of the innocent-looking sequence $(u_n)_{n \geq 0}$ given by 
\[
(n+1)(n^2-2) u_{n+1} = 2(2n+1)(n^2+2n-1) u_n, \quad u_0=1.
\]
Another example (see also \cref{ex:log}) is given by the two similar-looking hypergeometric functions 
\[
    f(x) = \pFq{2}{1}{1,1}{2}{x} \quad \text{ and } \quad g(x) = \pFq{2}{1}{2,2}{1}{x}.
\]
It is not difficult to see that $f(x) = -\log(1-x)/x$ and $g(x) = (1+x)/(1-x)^3$, hence $f(x)$ is clearly transcendental and $g(x)$ is trivially algebraic. However, the interlacing criterion of Beukers and Heckman can neither be applied  to $f(x)$ nor to $g(x)$, since both functions have top and bottom parameters with integral difference. From the viewpoint of our main \cref{thm:tree} the relevant distinction between $f(x)$ and $g(x)$ is that $f(x)$ is already \emph{contracted} but not \emph{reduced} (both terms to be defined in \cref{sec:result}), while the \emph{contraction} of $g(x)$ is given by the \emph{reduced} function $g^{\cc}(x) = {}_{1}F_{0}([1],[\;];x) = (1-x)^{-1}$.

Like in the examples above, one approach to overcome the hindrances on concrete examples is to use some of the many classical transformation formulas for hypergeometric functions and translate the problem to other functions, for which the criterion of Beukers and Heckman applies. In this paper we describe an algorithmic criterion that reduces the question of algebraicity for any hypergeometric function with complex parameters to the criterion of Beukers and Heckman, making these ad-hoc methods obsolete. Also, we will make use of a criterion in similar fashion, \cite[Prop.~1]{Christol86}, which essentially characterizes integer hypergeometric sequences (see Theorem \ref{thm:Christol}). \\

In \cref{sec:history} we give a detailed historical overview on the various contributions to the question of algebraicity of hypergeometric functions. Then, we will present the complete algorithm in Section~\ref{sec:result} and its proof thereafter, in Section~\ref{sec:proof}. Finally we provide several illustrating examples in Section~\ref{sec:examples}.

\paragraph{Acknowledgements.} We are deeply indebted to Alin Bostan for his support, various illuminating discussions not only regarding algebraic hypergeometric functions, and several crucial references which he pointed out to us. Further, we would like to express our gratitude for his numerous comments on an earlier version of this text. We also thank Herwig Hauser, Florian Lang and Stephan Schneider for many discussions and suggestions concerning this work. Moreover, we are very grateful to the anonymous reviewers for their comments and constructive suggestions that significantly improved our exposition (specifically the formulation of the main 
\cref{thm:tree} and the proofs of \cref{lem:span}, \cref{cor:falgfcalg} were greatly improved thanks to them). Finally, we want to thank Juan Arias de Reyna for writing the blog entry \href{https://institucional.us.es/blogimus/en/2024/01/flowers-of-the-hypergeometric-garden/}{Flowers of the hypergeometric garden} on the IMUS Blog of the University of Sevilla about the preprint version of this paper.

The authors were funded by the project P34765 of the \href{https://www.fwf.ac.at/en}{Austrian science fund FWF} and the second author was supported by the ANR-19-CE40-0018 \href{https://specfun.inria.fr/chyzak/DeRerumNatura/}{De Rerum Natura} as well as \href{https://oead.at/en/}{WTZ collaboration}/\href{https://www.campusfrance.org/}{Amadeus project} FR-09/2021 (46411YJ). The second author also thanks \href{https://www.artech.at/}{A\&R TECH} for (financial) support. 

\section{Previous work} \label{sec:history}

In the case of \emph{Gaussian hypergeometric functions}, i.e., if $p=2$ and $q = 1$ in~(\ref{def:hyper}) and $F(x) = {}_{2}F_{1}([\alpha,\beta],[\gamma];x)$ for some $\alpha,\beta,\gamma \in \Q$, the algebraicity was first classified by Schwarz~\cite{Schwarz73} as early as in 1873. A detailed explanation of Schwarz' method and result can be found in Klein's lecture notes~\cite[\S 57]{Klein81}, the book by Poole~\cite[\S VII]{Poole60} or lecture notes by Matsuda~\cite[\S I]{Matsuda85}. In order to apply Schwarz' classification, one first defines the values $\lambda,\mu,\nu$ as $1-\gamma,\gamma-\alpha-\beta,\beta-\alpha$ up to permutations, sign changes of some of the values, and addition of $(\ell,m,n)\in\Z^3$ such that $\ell+m+n$ is even. Then, assuming that $\lambda,\mu,\nu \not \in \Z$, it holds that $F(x)$ is algebraic if and only if $\lambda,\mu,\nu$ appears in Schwarz' list~\cite[p.323]{Schwarz73}. 
\begin{ex} \label{ex:Gessel0}
    Consider the function $f(x) = {}_{2}F_{1}([-1/2, -1/6],[2/3];x)$ which is directly related to Gessel excursions (see \cref{ex:Gessel1}). Then $(\lambda,\mu,\nu) = (1/3,4/3,1/3)$. Note that neither this triple nor any permutation of $(\pm \lambda, \pm \mu,\pm \nu)$ is in Schwarz' list. However, a permutation of $(\lambda,\mu-1,-(\nu-1)) = (1/3,1/3,2/3)$ corresponds to case III of that table, thus $f(x)$ is algebraic. 
\end{ex}

We remark that the 12 cases $1^a$--$4^c$ in \S I of Schwarz' work deal with the non-trivial ``reducible'' case: if one of the values $\alpha-\gamma,\beta-\gamma,\alpha,\beta$ is an integer, or equivalently, if the corresponding hypergeometric differential operator factors as the product of two first-order operators. We obtain these cases in \cref{cor:2F1list} below.

On the other hand, there is a very interesting and instructive arithmetic approach by Landau for the classification of algebraic Gaussian hypergeometric functions in~\cite{Landau04, Landau11}. Landau exploits Eisenstein's theorem (stated in \cite{Eisenstein52} by Eisenstein and first proved by Heine in~\cite{Heine1853, Heine1854}):
\begin{thm}[Eisenstein's theorem] \label{thm:Eisenstein}
    If $f(x) \in \Q\ps{x}$ is algebraic then there exists a non-zero integer $M \in \Z$ such that $f(Mx) - f(0) \in \Z\ps{x}$.
\end{thm}

We shall call a power series $f(x)$ that fulfills Eisenstein's criterion, that is if $f(Mx) - f(0) \in \Z\ps{x}$ for some $M \in \Z \setminus \{ 0 \}$, \emph{almost integral}. 
It follows from \cref{thm:Eisenstein} that if the Gaussian hypergeometric function $F(x)$ is algebraic then for all~$n \geq 0$
\begin{equation} \label{eq:gaussian-coeffs}
\frac{(\alpha)_n (\beta)_n}{(\gamma)_n n!} \in \Z[1/M] 
\end{equation}
for some $M \in \Z \setminus \{0\}$. In other words, only finitely many distinct prime numbers appear in the denominators of (\ref{eq:gaussian-coeffs}). Assuming that $\alpha,\beta,\gamma \in \Q$ but $\alpha-\gamma,\beta-\gamma,\alpha,\beta,\gamma \not \in \Z$, Landau could show that this is precisely the case if for all $\lambda \in \Z$ coprime to the common denominator of $\alpha,\beta,\gamma$ either
\begin{equation} \label{eq:Landau-interlacing}
\dpc{\lambda \alpha} < \dpc{\lambda \gamma} < \dpc{\lambda \beta} \text{ or } \dpc{\lambda \beta} < \dpc{\lambda \gamma} < \dpc{\lambda \alpha},
\end{equation}
where $\dpc{x} \coloneqq x - \lfloor x \rfloor$ for $x \in \Q \setminus \Z$ and $\lfloor x \rfloor$ denotes the floor function. In other words, $\exp(2 \pi i \lambda \gamma )$ is always between $\exp(2 \pi i \lambda \alpha )$ and $\exp(2 \pi i \lambda \beta)$ on the unit circle. Quite surprisingly, Errera~\cite{Errera13} could show that (\ref{eq:Landau-interlacing}) is also a sufficient criterion for algebraicity since it is equivalent to Schwarz' classification. Note that the assumption on the non-integrality of the parameters is of importance. For example, the famous complete elliptic integral of the first kind $K(x)$ can be written as 
\begin{equation}
    \frac{2}{\pi} K(\sqrt{x}) = \sum_{n \geq 0} \binom{2n}{n}^2 \left( \frac{x}{16}\right)^n = \pFq{2}{1}{1/2, 1/2}{1}{x} =1 + \frac{1}{4}x + \frac{9}{64}x^2 + \frac{25}{256}x^3 +   
\ldots, \label{eq:completeK}
\end{equation}
a Gaussian hypergeometric function with $\alpha=\beta=1/2$ and $\gamma=1 \in \Z$. The series is clearly almost integral, but not algebraic (see~\cite[\S2.1.2]{BoCaRo24} and more generally \cite{Christol86}).\\

The Landau-Errera criterion was rediscovered by Katz~\cite[\S 6]{Katz72} who not only reproved (\ref{eq:Landau-interlacing}) making use of his proof of the Grothendieck $p$-curvature conjecture for Picard-Fuchs equations, but also provided a conceptual reason why being almost integral is not only necessary, but also sufficient for hypergeometric functions. 

Now let us turn the attention to hypergeometric functions in general:
\[
F(x) = \pFq{p}{q}{a_1,\ldots, a_p}{b_1,\ldots, b_q}{x} \in \Q\ps{x}.
\]
We will for now assume that $F(x)$ is well-defined and not a polynomial, i.e., $a_j, b_k \not \in -\N$ for all $1 \leq j \leq p, 1 \leq k \leq q$. Then it is not difficult to see (\cref{lem:p=q+1} below) that algebraicity of $F(x)$ implies $p = q+1$. So we will also make this assumption.

One way to deal with irrational parameters is to invoke a theorem of Galo\v{c}kin~\cite[Thm.~2]{Galockin81} where he could prove that a hypergeometric function is a G-function (which is necessary for being algebraic) if and only if all its irrational parameters can be written in pairs $(a_{j_1},b_{k_1}),\dots,(a_{j_\ell},b_{k_\ell})$ with $j_{\mu} \neq j_{\nu}$ and $k_{\mu} \neq k_{\nu}$ for $\mu \neq \nu$ such that $a_{j_\nu} - b_{k_\nu} \in \N$ for $\nu=1,\dots,\ell$.  Rivoal recently gave a short proof of Galo\v{c}kin's theorem in~\cite{Rivoal22} which relies on a famously deep theorem of André, Chudnovsky and Katz that states that the local exponents of the minimal order differential equation of a G-function are rational numbers. Without relying on all these results, we will prove in an elementary way in Corollary~\ref{cor:falgfcalg} that $F(x)$ is algebraic if and only if $F^{\cc}(x)$ is algebraic, where $F^{\cc}(x)$ is obtained from $F(x)$ by removing all the pairs $(a_j,b_k)$ with $a_j-b_k \in \N$.

Therefore from now on in this section we will assume
\begin{equation} \label{eq:hypergeom:pFqrational}
F(x) = \pFq{p}{p-1}{a_1,\ldots, a_p}{b_1,\ldots, b_{p-1}}{x} \in \Q\ps{x}, \quad a_1\dots,a_p,b_1,\dots,b_{p-1} \in \Q \! \setminus \! ( -\N ).
\end{equation}
In the amazing but not so well-known work~\cite{Christol86}, Christol precisely characterizes such hypergeometric functions that are almost integral. Following his work, we define
\[
\langle x \rangle \coloneqq \begin{cases} 1 & \text{ if } x\in \Z,\\
x - \lfloor x \rfloor & \text{ else},
\end{cases}
\]
and the following total ordering on $\R$: 
\begin{equation}
a\preceq b \longeq \dpc{a}<\dpc{b} \text{ or } (\dpc{a}=\dpc{b} \text{ and } a\geq b),
\end{equation}
in order to deal with the case $a_j - b_k \in \Z$ for some $j,k$, i.e., $\langle a_j \rangle = \langle b_k \rangle$. Note that for two distinct real numbers with the same fractional part, the smaller one in $\R$ is greater with respect to $\preceq$. Then Christol proves the following theorem \cite[Prop.~1]{Christol86}:

\begin{thm}[Christol]\label{thm:Christol}
Let $F(x)$ be as in (\ref{eq:hypergeom:pFqrational}), denote by $N$ the least common denominator of all parameters, and set $b_p=1$. Then $F(x)$ is almost integral if and only if for all $1\leq \lambda  \leq N$ with $\mathrm{gcd}(\lambda ,N)=1$ we have for all $1 \leq k \leq p$ that 
\begin{equation} \label{eq:christolgeq0}
| \{\lambda a_j \preceq \lambda b_k \colon 1 \leq j \leq p\}|-| \{\lambda b_j \preceq \lambda b_k \colon 1 \leq j \leq p\}|\geq 0. 
\end{equation}
\end{thm}
Christol's criterion can be seen as a generalization of Landau's. On the unit circle it means that for each suitable $\lambda$ going counterclockwise starting after $1$, one encounters at least as many points $\exp(2\pi i \lambda a_j)$ as $\exp(2\pi i \lambda b_k)$. According to the definition of $\preceq$, in the case $\exp(2\pi i \lambda a_j) = \exp(2\pi i \lambda b_k)$ for some $j,k$, the point $\exp(2\pi i \lambda a_j)$ counts first if $ a_j > b_k$.\\


It is well known that $F(x)$ satisfies a linear homogeneous differential equation:
\begin{equation} \label{eq:hypergeomdiffeq}
x(\theta+a_1)\cdots (\theta+a_p)F(x) = \theta (\theta+b_1-1)\cdots (\theta+b_{p-1}-1)F(x) \qquad \big(\theta = x \frac{\mathrm{d}}{\mathrm{d}x} \big)
\end{equation}
which we will refer to as the \emph{hypergeometric equation}. We will prove in \cref{lem:mindifop} that in the case $a_j-b_k \not \in \N$ for all $j,k$ this equation is the least order homogeneous differential equation for $F(x)$. In this case the algebraicity of one solution of (\ref{eq:hypergeomdiffeq}) is equivalent to algebraicity of all its solutions (see, for example, \cite[Prop.~2.5]{Singer80}). If additionally to the assumptions $a_j - b_k \not \in \Z$ for all $j,k$ then a basis of solutions to the hypergeometric equation at zero is given by 
\begin{equation} \label{eq:basishypergeom}
z^{1-b_k} \pFq{p}{p-1}{1 + a_1 - b_k,\ldots, 1 + a_p - b_k}{1 + b_1 - b_k,\dots, 1 + b_{p} - b_k}{x}, \quad k=1,\dots,p,
\end{equation}
where $b_p \coloneqq 1$ and the parameter $1+b_k-b_k = 1$ at the bottom is omitted. 

Christol observed that under the above assumptions, if all these functions are almost integral, i.e., if \cref{thm:Christol} holds for all of them, then the left-hand side in (\ref{eq:christolgeq0}) attains only the values $0$ and $1$. On the unit circle this means that $\exp(2\pi i \lambda a_j)$ and $\exp(2\pi i \lambda b_k)$ actually ``interlace''. More precisely, one defines: 

\begin{defi} \label{defi:interlace}
    Given two equinumerous and disjoint multisets of real numbers $A=\{a_1,\ldots, a_p\}$ and $B=\{b_1,\ldots, b_p\}$, suppose that $\langle a_1\rangle\leq \langle a_2 \rangle \leq \cdots \leq \langle a_p \rangle $ and $\langle b_1 \rangle \leq \langle b_2 \rangle \leq\cdots \leq \langle b_p \rangle$. We say that $A$ and $B$ {\it interlace}, if $\langle a_1 \rangle < \langle b_1 \rangle < \langle a_2 \rangle < \cdots < \langle a_p \rangle < \langle b_p \rangle.$ 
\end{defi}
As before, interlacing according to \cref{defi:interlace} can be graphically interpreted by looking at the (multi)sets $\{ \exp(2\pi i a_j)\}$ and $\{\exp(2\pi i b_k)\}$ on the unit circle. 

The above considerations imply the following corollary to \cref{thm:Christol}:
\begin{cor}[Christol] \label{cor:christol:alg}
    Assume that $F(x)$ is as in (\ref{eq:hypergeom:pFqrational}) and $a_j, a_j-b_k \not \in \Z$. If $F(x)$ is algebraic then $A=\{\lambda a_1,\ldots, \lambda a_p\}$ and $B=\{\lambda b_1,\ldots, \lambda b_{p-1}, \lambda\}$ interlace for all $1\leq \lambda  \leq N$ with $\mathrm{gcd}(\lambda ,N)=1$, where $N$ is the least common denominator of the parameters.
\end{cor}
As a consequence of Katz's proof~\cite{Katz72} of the Grothendieck $p$-curvature conjecture for Picard-Fuchs differential equations, algebraicity of a basis of solutions of the hypergeometric equation is \emph{equivalent} to all solutions lying in $x^\rho \Z[\frac{1}{N}]\ps{x}$ for some $\rho \in \Q$ and $N \in \Z$ non-zero. Therefore the condition given in \cref{cor:christol:alg} is actually also sufficient for algebraicity. This reasoning was explained for Gaussian hypergeometric functions in~\cite[\S VI]{Katz72} and for general hypergeometric functions in~\cite[Chap. 5]{Katz90}, but apparently missed by Christol. Independently, Beukers and Heckman proved the same interlacing criterion in~\cite[Thm.~4.8]{BH89} by constructing a Hermitian form invariant under the monodromy of the hypergeometric equation, and consequently proving that the monodromy group is contained in the unitary group $U_n(\C)$ if and only if the parameters interlace.

\begin{thm}[Christol, Beukers-Heckman, Katz] \label{thm:BeHe} Let $F(x)$ be a hypergeometric function as in (\ref{eq:hypergeom:pFqrational}) and assume that $a_j, a_j - b_k \not \in \Z$ for all $j,k$. Let $N$ be the least common denominator of the parameters. Then $F(x)$ is algebraic if and only if $A=\{\lambda a_1,\ldots, \lambda a_p\}$ and $B=\{\lambda b_1,\ldots, \lambda b_{p-1}, \lambda\}$ interlace for all $1\leq \lambda  \leq N$ with $\mathrm{gcd}(\lambda ,N)=1$.
\end{thm}

\begin{ex} \label{ex:BeHe}
    In order to decide whether $f(x) = {}_{3}F_{2}([1/14,3/14,11/14],[1/7,3/7];x)$ is algebraic we set $(a_1,a_2,a_3) = (1/14,3/14,11/14)$ and $(b_1,b_2,b_3) = (1/7,3/7,1)$ and draw for each $1 \leq \lambda \leq 14$ coprime to 14, i.e., for $\lambda \in \{1,3,5,9,11,13 \}$, on the unit circle the points $\{\exp(2 \pi i \lambda a_j): j = 1,2,3\}$ in {\color{red}red} and the points $\{\exp(2 \pi i \lambda b_k): k = 1,2,3\}$ in~{\color{blue} blue}:
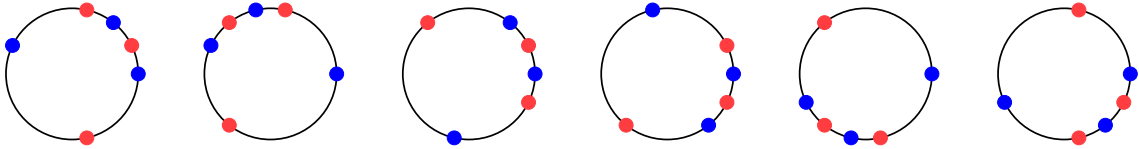
        \[
        \BeHe{14}{1/14}{3/14}{11/14}{1/7}{3/7}
        \]
    In each case the points interlace, therefore one concludes that $f(x)$ is an algebraic function.
\end{ex}


We note that Beukers and Heckman provided a list~\cite[\S5, Thm.~7.1, Table~8.3]{BH89}, similar to the one by Schwarz one century earlier but for ${}_{p}F_{p-1}$ with $p \geq 3$, in which they list all possible sets of parameters (under the assumptions of the theorem above), such that the criterion for algebraicity is fulfilled. \cref{ex:BeHe} above corresponds to the case No.~4 in their table. 

It is the purpose of the present work to investigate what can be said if in \cref{thm:BeHe} one drops the assumptions $a_j,b_k \in \Q$ and $a_j, a_j - b_k \not \in \Z$ for all $j,k$.

\section{The algebraicity criterion} \label{sec:result}
In this section we describe a full decision procedure which resolves the question whether a hypergeometric function is algebraic. We shall fix some notation first: Let 
\[
F(x) = \pFq{p}{q}{a_1,a_2,\ldots, a_p}{b_1,b_2,\ldots, b_{q}}{x} 
\coloneqq \sum_{n=0}^\infty \frac{(a_1)_n\cdots (a_p)_n}{(b_1)_n\cdots (b_q)_n}  \frac{x^n}{n!}
\]
be a hypergeometric function. We will use the three notions \emph{reduced}, \emph{contracted} and \emph{contraction}. It turns out that these terms are natural for 
\[
\F(x) = \bigF{c_1,c_2,\ldots, c_r}{d_1,d_2,\ldots, d_{s}}{x} \coloneqq \sum_{n=0}^\infty \frac{(c_1)_n\cdots (c_r)_n}{(d_1)_n\cdots (d_s)_n}  x^n,
\]
but look like quite strange definitions for $F(x)$. Note that via
\begin{align*}
    \bigF{c_1,c_2,\ldots, c_r}{d_1,d_2,\ldots, d_{s}}{x} & = \pFq{p}{q}{c_1,c_2,\ldots, c_r,1}{d_1,d_2,\ldots, d_{s}}{x}  \text{ and } \\
    \pFq{p}{q}{a_1,a_2,\ldots, a_p}{b_1,b_2,\ldots, b_{q}}{x} & = \bigF{a_1,a_2,\ldots, a_p}{b_1,b_2,\ldots, b_{q},1}{x}
\end{align*}
 each $F(x)$ can be rewritten as $\F(x)$ and vice versa. Therefore, every time there is a statement for $F(x)$ we implicitly mean it for the corresponding $\F(x)$ and also the other way around. For the following definitions we assume that $F(x)$ and $\F(x)$ are well-defined and not polynomials, i.e., all the parameters $a_i, b_j, c_k, d_\ell$ are not in $-\N$. Moreover, we will always assume $a_j\neq b_k$ and $c_j\neq d_k$ for all pairs~$(j,k)$.

We say that $\F(x)$ is \textit{reduced}, if $c_j - d_k\not \in \Z$ for all $1\leq j\leq r, 1\leq k\leq s$. This is exactly Christol's notion ``\textit{réduite}'', see \cite[p.12]{Christol86}. For $F(x)$ this means that $a_j - b_k\not \in \Z$ for all $1\leq j\leq p, 1\leq k\leq q$ and at most one $a_j \in \Z$ which then must be equal to 1.


Further, $\F(x)$ is said to be \textit{contracted} if $c_j-d_k\not \in \N$ for all $1\leq j\leq r, 1\leq k\leq s$. We will say that $F(x)$ is contracted if the corresponding $\F(x)$ is. Clearly, being contracted is a strictly weaker property than being reduced. 

\begin{ex} \label{ex:reduced-contracted}
Consider the hypergeometric function \( f(x) = {}_{2}F_1([1/2, 2],[3/2];x).
    \) Clearly, $f(x)$ is not reduced, since $1/2-3/2 \in \Z$ and also $2 \in \Z$. Moreover, $f(x)$ is not contracted either since $f(x) = \F([1/2, 2],[3/2,1];x)$ and $2-1 \in \N$. On the other hand, the function $g(x) = {}_{2}F_1([1/2, 1/2],[3/2];x)$ is contracted but not reduced. Finally, $h(x) = {}_{2}F_1([1/2, 1],[1/3];x)$ is neither reduced nor contracted since $h(x) = \F([1/2],[1/3];x)$.
\end{ex}

Finally, we define the \textit{contraction} $\F^{\cc}(x)$ of $\F(x)$ as the function obtained from $\F(x)$ by removing in each step one of the pairs of parameters $c_j, d_k$, where $c_j - d_k \in \N$ and $c_j - d_k$ minimal among such differences, until no such pair exists any more. Note that the contraction is independent of the order of removal of pairs of parameters with the same difference. Moreover, the contraction of $\F(x)$ is contracted by definition and the contraction of a contracted function is the function itself. As before, the contraction of a hypergeometric function $F(x)$ is defined as the contraction of the corresponding $\F(x)$.

\begin{ex}
The contraction of $f(x) = {}_{2}F_1([1/2, 2],[3/2];x)$ in \cref{ex:reduced-contracted} is $f^{\cc}(x) = \F([1/2],[3/2];x) = {}_{2}F_1([1/2, 1],[3/2];x)$ which is not reduced. Note that on the level of the $_{2}F_{1}$ functions one only sees that the parameter 2 turned into 1 in the process of contraction, even though actually the pair (2,1) was removed in $\F$ and then the artificial $1$ was added. This is the reason we defined the contraction via $\F(x)$. Now let $g(x)$ be the hypergeometric function in (\ref{eq:hyperintro}), i.e.
    \[ g(x) = \pFq{3}{2}{1/2,,\sqrt{2}+1,,-\sqrt{2}+1}{\sqrt{2},,-\sqrt{2}}{x} = \bigF{1/2,,\sqrt{2}+1,,-\sqrt{2}+1}{\sqrt{2},,-\sqrt{2},,1}{x}.
    \]
    Clearly, $g(x)$ is neither reduced, nor contracted and the contraction $g^{\cc}(x)$ is given by $\F([1/2],[1];x) = {}_{1}F_{0}([1/2],[\;];x) = (1-x)^{-1/2}$.
\end{ex}

\begin{ex}
    In the contraction steps the assumption on minimality between the differences $a_j-b_k$ is important as the following example shows:
    \begin{equation} \label{eq:excontr}
        \bigF{\frac{1}{3}, \frac{1}{2}, 2, 4}{\frac{3}{2}, 3, 1, 1}{x}^{\cc} = \bigF{\frac{1}{3}, \frac{1}{2}}{\frac{3}{2},1}{x}.
    \end{equation}
    If we first removed the pair $(a_4,b_3) = (4,1)$, we would obtain a different contracted function. Note that the function in (\ref{eq:excontr}) is contracted but not reduced.
\end{ex}

\begin{defi} \label{defi:ICfunc}
    Let $\F(x) = \F([c_1,\dots,c_r],[d_1,\dots,d_s];x)$ be a hypergeometric function with rational parameter sets $C = \{c_1,\dots,c_r\}$ and $D = \{d_1,\dots,d_s\}$ and let $N$ be the least common denominator of all the parameters. We say that $\F(x)$ satisfies the \emph{interlacing criterion (IC)} if $r=s$ and for all $1 \leq \lambda \leq N$ with $\gcd(\lambda,N)=1$ the multisets $\lambda C$ and $\lambda D$ interlace (according to \cref{defi:interlace}). 
\end{defi}

We mention that \cref{thm:BeHe} precisely says that if a hypergeometric function $F(x)={}_pF_{p-1}([a_1,\ldots, a_p],[b_1,\ldots, b_{p-1}];x)$ fulfills the assumptions of that theorem, that is, $a_j, a_j - b_k \in \Q \setminus \Z$ for all $j,k$, then $F(x)$ is algebraic if and only if it satisfies the interlacing criterion~IC.\footnote{Note that transcendental but almost integral hypergeometric functions still exist (see (\ref{eq:completeK})) -- those must, however, have at least one integral parameter~\cite[\S VI]{Christol86}.} If, however, $a_j \in \Z$ or $ a_j - b_k \in \Z$ for some $j,k$ (i.e., if $F(x)$ is not reduced), then the interlacing criterion automatically cannot be fulfilled, because Definition~\ref{defi:interlace} of interlacing has strict inequalities.  Also remark that the interlacing criterion can, by definition, only hold if $p=q+1$.

We now formulate our criterion of algebraicity for hypergeometric functions without any restriction on the parameter set. It reduces the question to the case in which \cref{thm:BeHe} is applicable, i.e., reduced hypergeometric functions with rational parameters.
\begin{thm}\label{thm:tree}
    A non-polynomial hypergeometric function $F(x) \in \Q\ps{x}$ is algebraic over $\Q(x)$ if and only if its contraction $F^\cc(x)$ has rational parameters and satisfies the interlacing criterion (IC).
\end{thm}
An equivalent version of this theorem is presented in Figure~\ref{fig:tree} which can be used to decide conveniently whether any given hypergeometric function is algebraic. 




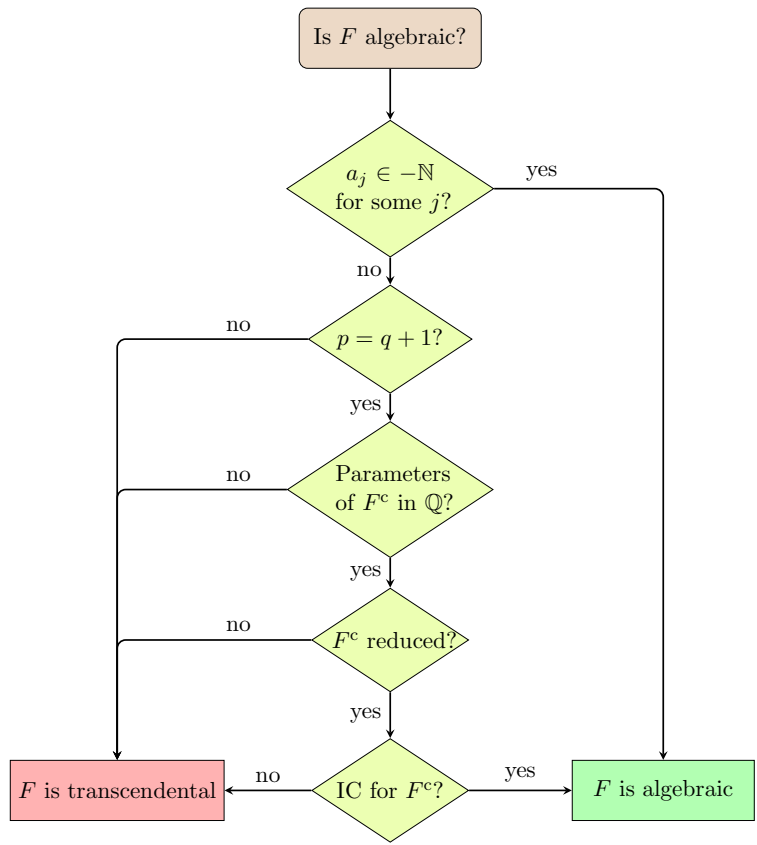
\begin{figure}[h]
    \centering
        \resizebox{10cm}{!}{
        \begin{tikzpicture}[node distance=2.5cm]
        \node (start) [startstop] {Is $F$ algebraic?};
        \node (dec1) [decision, below of=start] {$a_j \in -\N$ \mbox{for some $j$?}};
        \node (dec2) [decision, below of=dec1] {$p=q+1$?};
        \node (dec3) [decision, below of=dec2] {Parameters \mbox{of $F^{\cc}$ in $\Q$?}};
        \node (dec4) [decision, below of=dec3] {\mbox{\!$F^{\cc}$ reduced? }};
        \node (dec5) [decision, below of=dec4] {IC for $F^{\cc}$?};
        \node (no) [processno, left of=dec5, xshift=-2cm] {$F$ is transcendental};
        \node (yes) [processyes, right of=dec5, xshift=2cm] {$F$ is algebraic};
        \draw [arrow] (start) -- (dec1);
        \draw [arrow] (dec1) -- node[anchor=east] {no} (dec2);
        \draw [arrow] (dec2) -- node[anchor=east] {yes} (dec3);
        \draw [arrow] (dec3) -- node[anchor=east] {yes} (dec4);
        \draw [arrow] (dec4) -- node[anchor=east] {yes} (dec5);
        \draw [arrow] (dec5) -- node[anchor=south] {no} (no);
        \draw [arrow] (dec5) -- node[anchor=south] {yes} (yes);
        \draw [arrow] (dec1) [rounded corners] -|node[xshift=-2cm, anchor = south]  {yes} (yes);
        \draw [arrow] (dec2) [rounded corners] -|node[xshift=2cm, anchor = south] {no} (no);
        \draw [arrow] (dec3) [rounded corners] -|node[xshift=2cm, anchor = south] {no} (no);
        \draw [arrow] (dec4) [rounded corners] -|node[xshift=2cm, anchor = south] {no} (no);
        \end{tikzpicture}
        }
    \caption{Deciding whether $F(x)={}_pF_q([a_1,\ldots a_p],[b_1,\ldots, b_q];x)$ is algebraic: If $a_j\in -\N$ for some $j$, then $F(x)$ is a polynomial and thus algebraic. Otherwise only those functions $F(x)$ with $p=q+1$, where the contraction $F^{\cc}(x)$ is reduced and has rational parameters may be algebraic. It remains to check whether $F^{\cc}(x)$ satisfies the interlacing criterion~(IC) (\cref{defi:ICfunc}).}
    \label{fig:tree}
\end{figure}

\begin{rem}
    Theorem~\ref{thm:tree} is stated for hypergeometric functions $F(x)\in \Q\ps{x}$ with coefficients in $\Q$. Actually, it is possible to relax this requirement and allow $F(x)\in \overline{\Q}\ps{x}$ with coefficients being algebraic numbers, and the same statement holds true. The proofs of Lemma~\ref{lem:poly} through Proposition~\ref{prop:notrational} work analogously in this case. However, if one allows some of the coefficients (or, equivalently, the parameters) of $F(x)$ to be in $\C \setminus \overline{\Q}$, algebraicity over $\Q(x)$ can no longer be achieved. Still, one obtains the following assertion: Assume that the coefficients of $F(x)$, or, equivalently by Lemma~\ref{lem:algnumber}, the parameters of $F(x)$, lie in the algebraically closed transcendental field extension $K\supseteq \Q$. Then $F(x)\in K\ps{x}$ is algebraic over the field $K(x)$ if and only if its contraction $F^\cc(x)$ has rational parameters and satisfies the interlacing criterion (IC). For example, the hypergeometric function \[F(x)=\pFq{2}{1}{\pi+1, 1}{\pi}{x}=\sum_{n=0}^\infty \left(1+\frac{n}{\pi}\right)x^n = \frac{\pi + x - \pi x}{\pi (1 - x)^2}\] is a rational function in $\Q(\pi)(x)$. 
\end{rem}

For instance, \cref{thm:tree} allows us to classify all algebraic Gaussian hypergeometric functions \( {}_{2}F_{1}([\alpha, \beta],[\gamma];x) \).
We recall that this was already done by Schwarz~\cite{Schwarz73} (Artikel I for the ``reducible'' case and Artikel VI for the ``irreducible'' one). The latter case, that is, if $\alpha, \beta, \gamma-\alpha, \gamma-\beta \not \in \Z$ was additionally characterized by Landau in~\cite{Landau04} whose criterion is exactly the statement of \cref{thm:BeHe} for $p=2$ and $q=1$. Without loss of generality, two things can happen in the reducible case: Either there exists $k \in \Z$ such that $\gamma=\alpha+k$ or $\alpha = k \in \Z$.  The following exhaustive list characterizes algebraicity in these cases. We note that this criterion can also be derived from Vidūnas' work on degenerate Gaussian hypergeometric functions~\cite{Vidunas07}.  

\begin{cor}\label{cor:2F1list}
    The Gaussian hypergeometric function ${}_2F_1([\alpha,\beta],[\alpha+k];x)$ for $k \in \Z$ is algebraic if and only if either $k\leq0$ or: 
\begin{multicols}{2}
    \begin{enumerate}[a)]
        \item $\alpha\in \Z, \beta\not \in \Z$, or
        \item $\alpha\not \in \Z, \beta\in -\N$, or
        \item $\alpha,\beta\not \in \Z, \beta-\alpha-k\in \N$, or
        \item $\alpha, \beta\in \Z, \alpha<\beta\leq 0, \beta-\alpha\geq k$, or 
        \item $\alpha, \beta\in \Z, 0<\alpha<\beta\leq 0, \beta-\alpha\geq k$, or 
        \item $\alpha, \beta\in \Z, \beta\leq 0<\alpha$.
    \end{enumerate}
\end{multicols}
The function ${}_2F_1([k,\beta],[\gamma];x)$ for $k \in \Z$ is algebraic if and only if either $k \leq 0$ or: 
    \begin{multicols}{2}
    \begin{enumerate}[a)]
        \item $\beta \not \in -\N$, $\beta - \gamma \in \N$, or
        \item $\beta \in -\N$, or
        \item $\beta, \gamma \in \Z$, $0 < \beta < \gamma \leq k$, or
        \item $\beta \not \in \Z$, $k - \gamma \in \N$.
    \end{enumerate}
\end{multicols}
\end{cor}

We see that already in the case of non-reduced Gaussian hypergeometric functions the task of classifying algebraicity is cumbersome if it is done by case distinction.\\

As another corollary we also easily obtain the following interesting observation. 
\begin{cor} \label{cor:faf'a}
    Let $F(x)$ be a hypergeometric function. Then $F(x)$ is algebraic if and only if $F'(x)$ is algebraic.
\end{cor}
We will see that it is a special case of Lemma~\ref{lem:span} taking $\alpha=0$, but it can also be seen as an immediate consequence of Figure~\ref{fig:tree} in concert with the well-known identity 
\begin{equation} \label{eq:fdiff}
    F'(x)= \frac{a_1 \cdots a_p}{b_1 \cdots b_q } \pFq{p}{q}{a_1+1,\ldots, a_p+1}{b_1+1,\ldots, b_q+1}{x}.
\end{equation}


We remark that under the conditions of \cref{thm:BeHe}, \cref{cor:faf'a} is a trivial consequence of that theorem. If the assumption $a_j - b_k \not \in \Z$ is dropped but $a_j,b_k$ are still assumed to be rational, then \cref{cor:faf'a} can be concluded from \cref{thm:Christol} in combination with a theorem by André \cite[p.149]{Andre89} which states that if the primitive of an algebraic power series is almost integral (in the sense of \cref{thm:Eisenstein}), then it is algebraic (see also \cref{rem:andre}). Finally, if the assumption $a_j,b_k \in \Q$ is dropped as well (like in the general statement of \cref{cor:faf'a}), then one can additionally employ Galo\v{c}kin's characterization of hypergeometric G-functions ~\cite[Thm.~2]{Galockin81} (see also \cref{sec:history}) and conclude the same statement.

We also remark that the function $f(x) = -\log(1-x)$ may seem like a natural counter-example to \cref{cor:faf'a} since $f(x) = x \cdot {}_{2}F_{1}([1,1],[2];x)$, however $f(x)$ is not a quite a hypergeometric function (because of the multiplication by $x$) and $f(x)/x$ does not have the property that its derivative is algebraic.





\section{Proof of \cref{thm:tree}} \label{sec:proof}
This section is devoted to the proof of Theorem~\ref{thm:tree}. We first explain the strategy of the proof, afterwards we will state and prove all the auxiliary results used.

\begin{proof}[Proof of Theorem~\ref{thm:tree}.]
    We will check that the decision tree depicted in Figure~\ref{fig:tree} is correct. To see that it is equivalent to \cref{thm:tree} note that IC can only be fulfilled for a hypergeometric function with rational parameters if $p = q+1$ and if it is reduced. 
            
    The hypergeometric function 
    \[
        F(x)=\pFq{p}{q}{a_1,\ldots, a_p}{b_1,\ldots, b_q}{x}\in \Q\ps{x}
    \]
    is a polynomial if and only if $a_j\in -\N$ for some $j$ (Lemma \ref{lem:poly}), in this case $F(x)$ is clearly algebraic. Otherwise, $F(x)$ has to have algebraic parameters, due to Lemma~\ref{lem:algnumber}. By Lemma~\ref{lem:p=q+1} we must have $p=q+1$, or $F(x)$ is not algebraic. Now Corollary~\ref{cor:falgfcalg} shows that $F(x)$ is algebraic if and only if $F^{\cc}(x)$ is algebraic, thus we may restrict our study to the contraction of $F(x)$. Proposition~\ref{prop:notrational} shows that all the parameters of $F^{\cc}(x)$ need to be rational in order for $F(x)$ to be algebraic and Proposition~\ref{prop:notreduced} shows that $F^{\cc}(x)$ needs to be reduced in this case. Therefore, the question is simplified to deciding, whether a reduced hypergeometric function with rational parameters is algebraic. By \cref{lem:BetterHe} this is precisely the case if the interlacing criterion holds.
\end{proof}

In this section $F(x)$ always denotes the hypergeometric function
\[
F(x)=\pFq{p}{q}{a_1,\ldots, a_{p}}{b_1,\ldots, b_{q}}{x} = \sum_{n=0}^\infty \frac{(a_1)_n\cdots (a_p)_n}{(b_1)_n\cdots (b_q)_n}  \frac{x^n}{n!},
\]
where $a_1,\dots,a_p,b_1,\dots b_q \in \C$ are some parameters for which we assume that $b_k \not \in -\N$ for all $k = 1,\dots,q$. We will gradually impose more restrictions on the parameters. Moreover, 
\[
\F(x) = \bigF{c_1,c_2,\ldots, c_r}{d_1,d_2,\ldots, d_{s}}{x} \coloneqq \sum_{n=0}^\infty \frac{(c_1)_n\cdots (c_r)_n}{(d_1)_n\cdots (d_s)_n}  x^n,
\]
denotes the hypergeometric function, where we forego dividing the coefficients by $n!$, as introduced in the previous chapter. Here $c_1,\dots,c_r,d_1,\dots d_s \in \C$ are some parameters for which we assume that $d_k \not \in -\N$ for all $k = 1,\dots,s$. As mentioned before, we will naturally always assume that $a_j \neq b_k$ and $c_j \neq d_k$ for all pairs $(j,k)$.\\

The following statement is obvious but justifies the first step in Figure~\ref{fig:tree}: 
\begin{lem} \label{lem:poly}
    $F(x)$ is a polynomial if and only if $a_j \in -\N$ for some $j$. 
\end{lem}

Polynomials are trivially algebraic, henceforth we will also assume that $a_j,c_j \not \in -\N$. Now we show that the assumption that $F(x) \in \Q\ps{x}$ implies that all the parameters of the hypergeometric function are actually algebraic numbers.

\begin{lem} \label{lem:algnumber}
    If $F(x)\in\Q\ps{x}$ then $a_1,\dots,a_p,b_1,\dots b_q \in \overline{\Q}$ are algebraic numbers.
\end{lem}
\begin{proof}
    The recursion for the coefficient sequence $(u_n)_{n\geq0}$ of $F(x)$ reads as follows:
    \[
        u_{n+1} \prod_{i=1}^{q} (b_i+n)  =  u_n \prod_{i=1}^{p} (a_i+n).
    \]
    Hence, $R(z) \coloneqq \prod_{i=1}^{p} (a_i+z)/\prod_{i=1}^{q} (b_i+z) \in \C(z)$ has the property that $R(n) \in \Q$ for all $n \in \N$. From here it is easy to see that $R(z) \in \Q(z)$ by standard linear algebra arguments.
    Since the parameters $a_j$ and $b_k$ are the roots of the numerator and denominator of $R(z) \in \Q(z)$, they are algebraic numbers.
\end{proof} 

Recall that a function $f(x)=\sum_{n\geq 0}u_nx^n \in \Q\ps{x}$ is called {\it almost integral} if $f(Mx)-f(0) \in \Z\ps{x}$ for some $M>0$. Equivalently, there  are integers $\alpha, \beta\in \Z\setminus\{0\}$, such that $\beta f(\alpha x)\in \Z\ps{x}$. If in addition, the convergence radius of $f(x)$ is non-zero and finite, the power series is called {\it globally bounded} (see~\cite{Christol86}). Note that any almost integral function, which is not a polynomial, has finite radius of convergence. A classical theorem attributed to Eisenstein \cite{Eisenstein52}, \cref{thm:Eisenstein}, states that any algebraic power series in $\Q\ps{x}$ is almost integral (the first complete proof was given by Heine~\cite{Heine1853, Heine1854}). In particular, the convergence radius of a non-polynomial algebraic power series is finite and it is well-known that it is non-zero. 

\begin{lem} \label{lem:p=q+1}
  If $F(x) \in \Q\ps{x}$ is algebraic, then $p=q+1$.  
\end{lem}
\begin{proof}
For any $a\in\C \setminus (-\N)$ Stirling's formula implies that $(a)_n/n! \sim n^{1-a}/\Gamma(a)$ as $n \to \infty$, hence 
 \[
    u_n \sim \frac{\Gamma(b_1)\cdots\Gamma(b_q)}{\Gamma(a_1)\cdots\Gamma(a_p)}  n^{\sigma} (n!)^{p-q-1}, 
 \]
 where $u_n$ is the $n$-th coefficient of $F(x)$ and $\sigma = p-q-\sum_{j=1}^p a_j + \sum_{k=1}^q b_k$. Therefore, the radius of convergence $R \in \R \cup \{\infty\}$ of $F(x)$ computes to  
 \[R=\left( \limsup_{n\to \infty}|u_n|^{1/n} \right)^{-1}=\begin{cases}
     0 & p>q+1\\
     1 & p=q+1\\
     \infty & p<q+1.
 \end{cases}\]
With the remark from above this shows that $p=q+1$ is required.
\end{proof}

The Lemmata~\ref{lem:algnumber} and \ref{lem:p=q+1} allow us to assume from now on that $p=q+1$ and that all the parameters $a_1\dots,a_p, b_1\dots,b_{p-1}$ of $F(x)$ are algebraic numbers. \\

Write $\theta=x\partial=x\cdot \frac{{\rm d}}{{\rm d}x}$ for the Euler derivative and denote by  
\[
    H(\theta) = H([a_1,\dots,a_p],[b_1,\dots,b_{p-1}];\theta) \coloneqq  x \prod_{j=1}^{p} (\theta + a_j) - \theta \prod_{k=1}^{p-1} (\theta + b_k - 1),
\]
the \textit{hypergeometric differential operator}. Then it classically holds that
\begin{equation} \label{eq:hpergeom}
H([a_1,\dots,a_p],[b_1,\dots,b_{p-1}];\theta) \pFq{p}{p-1}{a_1,\dots,a_p}{b_1,\dots,b_{p-1}}{x} = 0. 
\end{equation}
The singularities of $H(\theta)$ are $0,1$ and $\infty$ and the operator is Fuchsian. Moreover, one can easily read off the local exponents at $x=0$ and $x=\infty$: the indicial polynomial at $x=0$ is given by $\chi_0(t)=t (t+b_1-1)\cdots (t+b_{p-1}-1)$ and consequently the local exponents are $0, 1-b_1,\ldots, 1-b_{p-1}$; the indicial polynomial at $x=\infty$ is $\chi_\infty(t)=(t-a_1)\cdots (t-a_p)$ and here the local exponents read $a_1,\ldots, a_p$. The local exponents at $1$ are given by $0,1,\dots,p-2$ and $\sum_{k=1}^{p-1} b_k-\sum_{j=1}^p a_j$ \cite{Pochhammer88}. 

 We will also use the following \emph{contiguous relation} which can be verified by a direct computation:
 \begin{align}\label{eq:wellknown}
 (\theta+a_1)F(x) &= a_1 \cdot \pFq{p}{p-1}{a_1+1,a_2, \ldots, a_p}{b_1, b_2, \ldots, b_{p-1}}{x}.
\end{align}

\begin{lem} \label{lem:span}
Assume that $F(x) = {}_p{F}_{p-1}([a_1,,\ldots, a_p],[b_1,\ldots, b_{p-1}];x)$ is hypergeometric and let $\alpha \in \overline{\Q} \setminus (\{a_1,\ldots, a_p\} \cap \{b_1-1,\ldots ,b_{p-1}-1, 0\})$. Then $F(x)$ is algebraic if and only if $(\theta+\alpha)F(x)$ is algebraic.
\end{lem}
\begin{proof}
    Clearly the forward implication holds. For the backward implication we use the Euclidean division of the differential operator $H(\theta)$ by $(\theta+\alpha)$. We find $Q(\theta),$ a differential operator of order $p-1$, such that
    \begin{equation} \label{eq:Euclid}
        H(\theta)=Q(\theta)(\theta+\alpha)+H(-\alpha), 
    \end{equation}
    where, by abuse of notation, we write 
    \[
        H(-\alpha)=x\prod_{j=1}^p(-\alpha+a_j) + \alpha \prod_{k=1}^{p-1}(-\alpha+b_k-1).
    \]
    Indeed, the Euclidean division works independently from commutativity (which is clearly not given in the ring of differential operators), thus performing the division by an order one operator $(\theta+\alpha)$, gives the (formal) evaluation of the operator $H$ at $-\alpha$ as remainder.

    Applying \eqref{eq:Euclid} to $F(x)$ gives $0=Q(\theta)(\theta+\alpha)F(x)+H(-\alpha)F(x)$. The polynomial $H(-\alpha)$ vanishes if and only if $\alpha\in \{a_1,\ldots, a_p\} \cap \{b_1-1,\ldots, b_{p-1}-1, 0\}$, therefore we obtain that for all other values of $\alpha$ the algebraicity of $(\theta+\alpha)F(x)$ implies that $F(x)$ is algebraic as well.
\end{proof}

The following corollary to \cref{lem:span} allows us to reduce the algebraicity question for hypergeometric functions to the contracted case. 
Recall that $F(x)$ is called \emph{contracted} if the corresponding $\F(x)$ is contracted which means that $c_j - d_k \not \in \N$ for all $j,k$. The contraction of $F(x)$, denoted by $F^{\cc}(x)$, is defined in the beginning of \cref{sec:result}.

\begin{cor} \label{cor:falgfcalg}
The function $F(x)$ is algebraic if and only if $F^{\cc}(x)$ is algebraic.
\end{cor}
\begin{proof}
We will show that algebraicity is preserved in each step of the contraction process. Write $F(x) = \F([c_1,\dots,c_r],[d_1,\dots,d_{r}];x)$. Assume without loss of generality that $c_1=d_1+\ell$ for some $\ell\in \N$ and this difference is minimal among pairs of parameters $c_j, d_k$ with $d_k-c_j\in \N$. We apply \cref{lem:span} together with relation \eqref{eq:wellknown} a total of $\ell$ times (choosing $\alpha = c_1 + i$ for $i=0,\dots,\ell-1$ consecutively) to show that 
\[
\bigF{d_1, c_2, \ldots, c_{r}}{d_1,d_2,\ldots, d_{r}}{x} = \bigF{c_2, \ldots, c_{r}}{d_2,\ldots, d_{r}}{x}
\]
is algebraic if and only if $F(x)$ is algebraic. Inductively we obtain that $F(x)$ is algebraic if and only if $F^{\cc}(x)$ is algebraic.
\end{proof}
\begin{rem} \label{rem:andre}
    For \cref{lem:span} and \cref{cor:falgfcalg} we had a cumbersome but still elementary proof and we thank the anonymous reviewer for the much more elegant proof presented above. Here is yet another viewpoint:
    
    André proved in \cite[p.149]{Andre89} that if the primitive of an algebraic power series is globally bounded, then it is algebraic. This statement in the setting of Puiseux series gives an alternative proof that a hypergeometric function $F(x)$ with rational parameters is algebraic if and only if $F^{\cc}(x)$ is algebraic. Indeed, assume that $(\theta + \alpha)F(x)$ is algebraic, then $x^{\alpha-1} (\theta+\alpha) F(x)=(x^{\alpha }F'(x)+\alpha x^{\alpha-1}F(x))=(x^{\alpha}F(x))'$ is algebraic as well, if $\alpha \in \Q$. One proceeds as in the proof of \cref{cor:falgfcalg} and uses that, by Christol's interlacing criterion (\cref{thm:Christol}) and (\ref{eq:wellknown}), $F(x)$ is globally bounded if and only if $(\theta+c_1+i)F(x)$ is globally bounded. Therefore $F(x)$ stays algebraic in each contraction step.
\end{rem}


\begin{lem} \label{lem:mindifop}
    Assume that $F(x)$ is contracted. Then the least order homogeneous differential equation satisfied by $F(x)$ is given by the hypergeometric differential equation (\ref{eq:hpergeom}). 
\end{lem}
The following proof is an adaptation of the proof of Lemma 4.2 in \cite{BBH88}.
\begin{proof}
    We have $F(x) = {}_{p}F_{p-1}([a_1,\dots,a_p],[b_1,\dots,b_{p-1}];x)$ and define $b_p \coloneqq 1$. Suppose that $F(x) = \sum_{n \geq 0} u_n x^n$ satisfies a linear differential equation of smaller order. Then there exists a recurrence relation for the sequence of coefficients $(u_n)_{n\geq0}$ of the form
    \[A_\ell(n)u_{n+\ell}+\ldots+ A_1(n)u_{n+1}+A_0(n)u_n=0,\]
    where $A_0(t), \dots, A_\ell(t) \in \Q[t]$ are polynomials with degree at most $p-1$. For any $i \geq 0$
    \[
    u_{n+i} = \frac{\prod_{j=1}^p(a_j+n)\cdots (a_j+n+i-1)}{\prod_{k=1}^{p}(b_k+n)\cdots (b_k+n+i-1)} u_n,
    \]
    thus we have
   \[
    \left(A_\ell(n)\frac{\prod_{j=1}^p(a_j+n)\cdots (a_j+n+\ell-1)}{\prod_{k=1}^{p}(b_k+n)\cdots (b_k+n+\ell-1)}+\ldots+A_1(n)\frac{\prod_{j=1}^{p}(a_j+n)}{\prod_{k=1}^{p}(b_k+n)}  +A_0(n)\right)u_n=0 ,
    \]
    for all $n \geq 0$. By assumption $F(x)$ is not a polynomial, thus $u_n\neq 0$ for infinitely many $n\in \N$. Therefore, we obtain the  identity of rational functions:
    \[
    A_\ell(t)\frac{\prod_{j=1}^p(a_j+t)\cdots (a_j+t+\ell-1)}{\prod_{k=1}^{p}(b_k+t)\cdots (b_k+t+\ell-1)}+\ldots+A_1(t)\frac{\prod_{j=1}^{p}(a_j+t)}{\prod_{k=1}^{p}(b_k+t)}  +A_0(t)=0 .
    \]
    We will show that for some $1 \leq \kappa \leq p$, the leftmost summand has at least one pole at $t=1-b_{\kappa}-\ell$ of a higher order than all the other summands. This clearly gives a contradiction. 
    
    For each $1 \leq k \leq p$ the order of the pole of 
    \[\frac{1}{\prod_{k=1}^{p}(b_k+t)\cdots (b_k+t+i-1)}\]
    at $t = 1-b_k-\ell$ for $i<\ell$ is strictly smaller than the order of that pole if $i=\ell$. We distinguish two cases: First, assume $a_{j}\neq 1$ for all $j$. Then, none of the factors $(b_k+t+\ell-1)$ from the denominator of the fraction can be canceled with a factor in the numerator, because if $a_j\equiv b_k\bmod \Z$, then $a_j<b_k$, as $F(x)$ is contracted by assumption. Moreover, at most $p-1$ of the factors $(b_k+t+\ell-1)$ might divide $A_\ell(t)$ because of its degree. This means that the first summand has a pole at one of the numbers $(1-b_k-\ell)$ of a higher order than all the other summands.
    
    Now assume that $a_m=1$ for some $m$, then $b_k\neq 1$ for $k\neq p$. The factor $(b_p+t+\ell-1)=(t+\ell)$ from the denominator in the first summand can be canceled with $(a_m+t+\ell-1)=(t+\ell)$ from the numerator. Note that $A_\ell(t-\ell)$ is the indicial polynomial of the differential equation. Thus $A_\ell(t)$ is divisible by $(t+\ell)$, as $0$ is a local exponent of any differential equation possessing a power series solution of order $0$ at the origin. This factor $(t+\ell)$ cannot cancel an additional factor of the denominator in the first summand, because $b_k\neq 1$ for $k\neq p$. Therefore the same reasoning as above applies and we arrive at the same conclusion.
\end{proof}

We can now show that a contracted hypergeometric function with irrational parameters cannot be an algebraic function.

\begin{prop} \label{prop:notrational}
    Let $F(x)$ be a contracted hypergeometric function. If $a_j \not \in \Q$ or $b_k \not \in \Q$ for some $j$ or $k$, then $F(x)$ is not algebraic.
\end{prop}
\begin{proof}
    According to Lemma \ref{lem:mindifop} the differential operator of minimal order that annihilates $F(x)$ is given by the hypergeometric one. A basis of solutions of the minimal order differential operator $L_f^{\mathrm{min}}$ annihilating an algebraic function $f(x)$ is given by the algebraic conjugates of $f(x)$, in particular all of the solutions of $L_f^{\mathrm{min}}$ are algebraic (this follows from elementary differential Galois theory, see, for example, \cite[Prop.~2.5]{Singer80} or \cite[\S2]{CoSiTrUl02}). The local exponents of (\ref{eq:hpergeom}) at $0$ are given by $0$ and  $1-b_k$ for all $k=1,\dots,p-1$, therefore the operator has at least one non-algebraic solution, if some $b_k\not \in \Q$. Similarly, the local exponents at $x=\infty$ read $a_j$ for $j=1,\dots,p$, therefore if some $a_j\not \in \Q$ we obtain a non-algebraic local solution at $x=\infty$. 
\end{proof}

\begin{rem}
    As mentioned in \cref{sec:history}, an alternative way to prove \cref{prop:notrational} is to employ a theorem by Galo\v{c}kin~\cite[Thm.~2]{Galockin81} which ensures that a contracted hypergeometric function with irrational parameters cannot be a G-function, in particular it cannot be algebraic.
\end{rem}

    
    


Recall that $F(x)$ is called \emph{reduced} if $a_j - b_k \not \in \Z$ for all $j,k$ and at most one $a_j \in \Z$ which is then equal to 1; equivalently, $\F(x)$ is reduced if $c_j-d_k \not \in \Z$. The following lemma relaxes the hypothesis of~\cref{thm:BeHe} slightly, allowing to decide algebraicity for any reduced hypergeometric function.
\begin{lem} \label{lem:BetterHe}
    Let $F(x)$ be a reduced hypergeometric function with rational parameters. Then $F(x)$ is algebraic if and only if it satisfies the interlacing criterion~(\cref{defi:ICfunc}).
\end{lem}
\begin{proof}
    Write $F(x) = \F([c_1,\dots,c_r],[d_1,\dots,d_r];x)$. We shall distinguish two cases, according to the occurrence of integer parameters. 

    If $d_m \in \Z$ for some $m \in \Z$ assume without loss of generality that $d_r \in \N$ and that each other integral $d_k$ is larger.  Since $\F(x)$ is reduced with rational parameters, it holds that $c_j,c_j - d_k \in \Q \setminus \Z$ for $j,k=1,\dots,r$. If now $d_r=1$ then $\F(x)={}_{r}F_{r-1}([c_1,\ldots, c_r],[d_1,\ldots, d_{r-1}];x)$, \cref{thm:BeHe} applies and concludes this case. If $d_r>1$, apply the trivial identity
    \begin{equation} \label{eq:trivid}
    \bigF{c_1-1,\dots,c_r-1}{d_1-1,\dots,d_r-1}{x} = 1 + \frac{(c_1-1)\cdots(c_r-1)}{(d_1-1)\cdots(d_r-1)} x \bigF{c_1,\dots,c_r}{d_1,\dots,d_r}{x}
    \end{equation}
    a total of $d_r-1$ times and conclude with the first case. Note that \eqref{eq:trivid} reduces all parameters by 1, hence interlacing is not affected.
    
    Now assume that $d_k\not \in \Z$ for all $k$. We have $F(x)={}_{r+1}F_r([c_1,\ldots, c_r, 1],[d_1,\ldots, d_r];x)$. 
    It is easy to see that the condition \eqref{eq:christolgeq0} cannot be fulfilled for all $1 \leq k \leq r$ and $\lambda = 1, N-1$ at once. Thus, by~\cref{thm:Christol}, $F(x)$ is not globally bounded, therefore by~\cref{thm:Eisenstein} it is not algebraic either. Condition \eqref{eq:christolgeq0} is necessary for the interlacing criterion, so in this case the criterion cannot be fulfilled as well.
\end{proof}

Similarly to the contraction we define the \textit{reduction} $\mathcal{F}^\rr(x)$ of $\mathcal{F}(x)$ as follows: remove from $\F^{\cc}(x)$ in each step one of the pairs of parameters $c_j, d_k$, where $d_k - c_j \in \N$ and $d_k - c_j$ is minimal among such differences, until no such pair exists anymore. Of course, $\mathcal{F}^\rr(x)$ is reduced. The reduction $F^{\rr}(x)$ of a hypergeometric function $F(x)$ is defined as the reduction of the corresponding function $\mathcal{F}(x)$.

\begin{lem}\label{lem:falgfralg}
    If $F(x)$ is algebraic, then so is $F^\rr(x)$. 
\end{lem}
\begin{proof}
    By Corollary~\ref{cor:falgfcalg}, $F^{\cc}(x)$ is algebraic. In each of the steps of the reduction of $F^{\cc}(x)$ algebraicity is preserved due to (\ref{eq:wellknown}).
\end{proof}

Finally we are able to proof that in order for $F(x)$ to be algebraic, its contraction $F^{\cc}(x)$ must be already reduced.

\begin{prop} \label{prop:notreduced}
Let $F(x)$ be an algebraic contracted hypergeometric function with rational parameters. Then $F(x)$ is reduced. 
\end{prop}
\begin{proof}
Assume that $F(x)$ is not reduced and write $F^\rr(x)=\F([c_1,\ldots, c_r],[d_1,\ldots, d_r];x)$. By Lemma~\ref{lem:falgfralg}, $F^\rr(x)$ is algebraic and it satisfies the hypothesis of  Lemma~\ref{lem:BetterHe}. Therefore, its sets of parameters $c_j$ and $d_k$ interlace on the unit circle.  Thus we can assume
\[
    \langle c_1\rangle < \langle d_1 \rangle < \langle c_2 \rangle <  \ldots < \langle c_r\rangle < \langle d_r\rangle.
\]
Note that exactly one parameter $d_k$ must be an integer in order for the interlacing criterion to be fulfilled. Indeed, if none of the parameters $d_k$ is integral, the criterion for $\lambda=N-1$ cannot be fulfilled, as this corresponds to going backwards on the unit circle. If two of the parameters $d_k$ are integers, the criterion is trivially not fulfilled. Thus, $d_r\in \N$. 

Let $c$ and $d$ be the two parameters removed last in the reduction process of $\F(x)$, in particular, $d-c\in \N$ and $d \preceq c$. Then, $G(x)=\mathcal{F}([c_1,\ldots, c_r, c],[d_1,\ldots, d_r, d];x)$ is contracted and it is also algebraic, since algebraicity is preserved in each reduction step. In particular, it is globally bounded by Eisenstein's Theorem, Theorem \ref{thm:Eisenstein}.

Because $F(x)$ is contracted there are two cases to consider: 
\[
    \langle d_k\rangle \leq \langle c\rangle=\langle d\rangle \leq \langle c_{k+1}\rangle \quad \text{ or } \quad \langle c_k\rangle < \langle c\rangle=\langle d\rangle < \langle d_{k}\rangle
\]
for some $k$. In the first case, if both inequalities are strict we clearly find 
\begin{equation} \label{eq:contraredcont}
    c_1\preceq \ldots\preceq c_k \preceq d_k\preceq d \preceq c \preceq c_{k+1}\preceq d_{k+1}\preceq \ldots\preceq d_r,
\end{equation}
which is a contradiction to the global boundedness of $G(x)$ in view of Theorem~\ref{thm:Christol}. Still in the same case, if $\langle d_k\rangle = \langle c\rangle=\langle d\rangle$ then $c\leq d\leq d_k$ because $F^\rr(x)$ is reduced and $F(x)$ contracted, and we again arrive at \eqref{eq:contraredcont}. Analogously, if $\langle c\rangle=\langle d\rangle = \langle c_{k+1}\rangle$ then $c_{k+1}\leq c\leq d$ because of the same reason and we find \eqref{eq:contraredcont} once more.

In the second case we obtain $c_k\preceq d \preceq c \preceq d_{k}$. But then 
\[
    \lambda c_r\preceq \ldots \preceq \lambda c_{k+1}\preceq \lambda d_k\preceq \lambda d \preceq \lambda c \preceq \lambda c_{k}\preceq \ldots \preceq \lambda c_1\preceq \lambda d_r\in \Z
\]
for $\lambda=N-1$. Note that $\lambda d \preceq \lambda c$ because $\lambda d - \lambda c \in \N$ and $\lambda d_r$ is last in this ordering, as it is an integer. As $G(x)$ is globally bounded, this is again a contradiction to Theorem~\ref{thm:Christol}.
\end{proof}

\section{Examples} \label{sec:examples}


In this section we aim to provide numerous examples that showcase different phenomena that can occur regarding algebraicity of hypergeometric functions. We start with the two examples from the introduction (\cref{sec:intro}):

\begin{ex} \label{ex:first}
    Recall the example from the introduction in which a sequence $(u_n)_{n \geq 0}$ is defined by the recursion 
    \[
    (n+1)(n^2-2) u_{n+1} = 2(2n+1)(n^2+2n-1) u_n, \quad u_0=1.
    \]
    By definition this corresponds to the following representation of the generating function:
    \begin{equation*}
        \sum_{n\geq0} u_n x^n = \pFq{3}{2}{1/2,,\sqrt{2}+1,,-\sqrt{2}+1}{\sqrt{2},,-\sqrt{2}}{4x}. 
    \end{equation*}
    Applying \cref{thm:tree} to $f(x) = {}_{3}F_2([1/2,\sqrt{2}+1,-\sqrt{2}+1],[\sqrt{2},-\sqrt{2}];x)$, we observe that its contraction $f^{\cc}(x) = {}_{1}F_{0}([1/2],[\,];x)$ has rational parameters and is reduced. Trivially, the interlacing criterion holds for $f^{\cc}(x)$, hence the generating function of $(u_n)_{n\geq 0}$ is algebraic.
\end{ex}

\begin{ex} \label{ex:log}
    Also recall the examples from the introduction, where we set
    \[
        f(x) = \pFq{2}{1}{1,1}{2}{x} = -\frac{\log(1-x)}{x} \quad \text{ and } \quad g(x) = \pFq{2}{1}{2,2}{1}{x} = \frac{1+x}{(1-x)^3}.
    \]
    Here $f(x)$ is clearly transcendental and $g(x)$ is trivially algebraic. \cref{thm:BeHe} does not apply for either of these functions since they have top and bottom parameters with integral difference. However, $f(x)$ is already contracted but not reduced (hence transcendental by \cref{thm:tree}) and the contraction of $g(x)$ is given by the reduced function $g^{\cc}(x) = {}_{1}F_{0}([1],[\;];x) = (1-x)^{-1}$ (thus $g(x)$ is algebraic by \cref{thm:tree}).
\end{ex}

The following example is constructed to illustrate the application of the interlacing criterion after the contraction step.

\begin{ex} \label{ex:crazy}
    Let $(u_n)_{n\geq 0}$ be defined via $u_0 = 1$ and 
    \[
    u_{n+1} = \frac{\left(14 n +1\right) \left(14 n +3\right) \left(14 n +11\right) \left(n^{2}+2 n +4\right)}{56 \left(7 n +1\right) \left(7 n +3\right) \left(n +3\right) \left(n^{2}+3\right)} u_n. 
    \]
    We are interested in the nature of the generating function $f(x) = \sum_{n\geq0} u_n x^n$. One finds 
    \[
        f(x) = \bigF{\frac{1}{14}, \frac{3}{14}, \frac{11}{14},  1+i\sqrt{3}, 1-i\sqrt{3}}{\frac{1}{7}, \frac{3}{7}, i\sqrt{3}, -i\sqrt{3},3}{x} = \pFq{6}{5}{\frac{1}{14}, \frac{3}{14}, \frac{11}{14},  1+i\sqrt{3}, 1-i\sqrt{3},1}{\frac{1}{7},, \frac{3}{7},, i\sqrt{3},, -i\sqrt{3},,3}{x}.
    \]
    Note that the result of Beukers and Heckman \cite[Thm.~4.8]{BH89} (see \cref{thm:BeHe}) does not apply directly for $f(x)$ since some parameters are irrational numbers and also have integer differences. 
    
    Following \cref{thm:tree}, we first compute the contraction 
    \begin{equation} \label{eq:crazy_cont}
    f^{\cc}(x)= \bigF{\frac{1}{14}, \frac{3}{14}, \frac{11}{14}}{\frac{1}{7},\frac{3}{7},3}{x} = \pFq{4}{3}{\frac{1}{14}, \frac{3}{14}, \frac{11}{14}, 1}{\frac{1}{7}, \frac{3}{7},3}{x},
    \end{equation}
    which now clearly has rational parameters and is reduced.  
    
    We remark that \cref{thm:BeHe} cannot be applied to $f^{\cc}(x)$ either, since in the ${}_{4}F_3$ representation some of the parameters have integral difference, but Lemma~\ref{lem:BetterHe} applies. If the parameters $1$ and $3$ are ignored (corresponding to the application of the transformation~\eqref{eq:trivid} twice) then the parameters of $f^{\cc}(x)$ correspond to case No. 4 in \cite[Table~8.3]{BH89}, which shows the algebraicity of~$f(x)$.  
    
    Alternatively, it suffices, according to Figure~\ref{fig:tree}, to check the interlacing criterion for $f^{\cc}(x)$. Similarly to \cref{ex:BeHe}, looking at \eqref{eq:crazy_cont} we set $(a_1,a_2,a_3) = (1/14,3/14,11/14)$ and $(b_1,b_2,b_3) = (1/7,3/7,3)$. Then the interlacing criterion applied to $f^{\cc}(x)$ tells us to consider the six images below: For each $1 \leq \lambda \leq 14$ coprime to 14 we draw on the unit circle the points $\{\exp(2 \pi i \lambda a_j): j = 1,2,3\}$ in {\color{myred}red} and the points $\{\exp(2 \pi i \lambda b_k): k = 1,2,3\}$ in~{\color{blue} blue}.
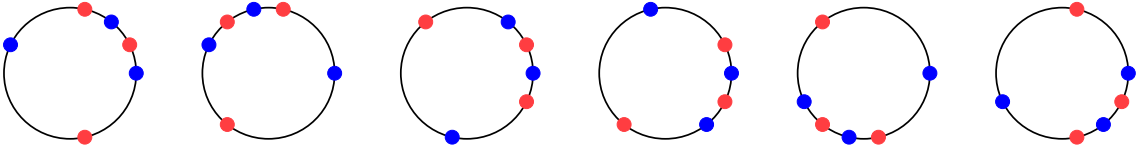
    \[
    \BeHe{14}{1/14}{3/14}{11/14}{1/7}{3/7}
    \]
    Since for each of the pictures the blue and red points interlace, the interlacing criterion for $f^{\cc}(x)$ is satisfied and $f^{\cc}(x), f(x)$ are algebraic functions.
\end{ex}

The hypergeometric sequence in the following example appears in \cite[\S4.3.3]{Yurkevich23} and is related to the isoperimetric ratio of a conformal transformation of a torus with minor radius $r=1$ and major radius~$R>1$.
\begin{ex}
    For some $R>1$ consider the sequence $(u_n(R))_{n \geq0}$ given by
    \[
        u_{n+1}(R) = \frac{(2n-1)(2n+1)\cdot p_{n+1}(R)}{4(n+2)(n+1) \cdot p_n(R)} u_n(R), \quad u_0(R)=1,
    \]
    where 
    \[
        p_n(R) = 4(R^4 + 4R^2 - 4)n^3 + 6(R^4 + R^2 - 2)n^2 + (2R^4 - 13R^2 + 10)n - 3R^2 + 3.
    \]
    If $R$ is considered as a variable, $p_n(R)$ is irreducible over $\Q[n,R]$. If we denote by $\alpha_1(R),\alpha_2(R),\alpha_3(R) \in \overline{\Q(R)}$ the roots of $p_n(R) = 0$ with respect to $n$, then the generating function of $(u_n(R))_{n \geq0}$ is given by
    \[
        f_R(x) = \sum_{n \geq 0} u_n(R) x^n = \pFq{5}{4}{-1/2,1/2,\alpha_1(R)+1,\alpha_2(R)+1,\alpha_3(R)+1}{2,\alpha_1(R),\alpha_2(R),\alpha_3(R)}{x}.
    \]
    Note that for any $R \in \Q$ the series $f_R(x)$ is globally bounded. The contraction of $f_R(x)$ is given by 
    \[
        f_R^{\cc}(x) = \pFq{2}{1}{-1/2,1/2}{2}{x},
    \]
    which has rational parameters and is reduced but is not algebraic (does not satisfy IC). It follows that also $f_R(x)$ is never algebraic. 

    On the other hand, if we switch the top parameter $-1/2$ with bottom parameter $2$, i.e., consider the function
    \[
        g_R(x) = \sum_{n \geq 0} u_n(R) x^n = \pFq{5}{4}{1/2,2,\alpha_1(R)+1,\alpha_2(R)+1,\alpha_3(R)+1}{-1/2, \alpha_1(R),\alpha_2(R),\alpha_3(R)}{x},
    \]
    then $g_R^{\cc}(x) = \F([\,],[\,];x) = {}_{1}F_0([1],[\,];x) = 1/(1-x)$ which is trivially algebraic. Therefore, $g_R(x)$ is an algebraic function for any $R \in \Q$, even though it has irrational parameters.
\end{ex}

\begin{ex}
    Consider the sequence $(u_n)_{n \geq 0}$ given by the closed form expression 
    \begin{equation} \label{eq:uex}
        u_n = \frac{3}{2} \binom{4n}{n} \frac{n+2}{(n+1)(n+3)}.
    \end{equation}
    For its generating function $f(x) = \sum_{n\geq0} u_n x^n$ one finds that
    \[
        f(x) = \pFq{6}{5}{1/4, 1/2, 3/4, 3, 3, 1}{1/3,, 2/3,, 4,, 2,, 2}{\frac{256}{27}x} = 1+\frac{9}{4} x +\frac{56}{5} x^{2}+\frac{275}{4} x^{3} + \cdots.
    \]
    The contraction $f^{\cc}(x)$ is given by removing twice the pair $(3,2)$ from the ${}_{6}F_5$ above, i.e.,
    \[
        f^{\cc}(x) = \pFq{4}{3}{1/4, 1/2, 3/4, 1}{1/3,, 2/3,, 4}{\frac{256}{27}x} = \bigF{1/4, 1/2, 3/4}{1/3,, 2/3,, 4}{\frac{256}{27}x}.
    \]
    The interlacing criterion applied to the hypergeometric function with the parameters $(a_1,a_2,a_3) = (1/4,1/2,3/4)$ and $(b_1,b_2,b_3) = (1/3,2/3,4)$ shows\footnote{Alternatively note that this is case No. 1 in \cite[Table~8.3]{BH89}} that $f^{\cc}(x)$ and consequently $f(x)$ are algebraic functions.

    Now consider the sequence $(v_n)_{n \geq 0}$ given by
    \begin{equation}
        v_n = \frac{1}{2} \binom{4n}{n} \frac{n+2}{(n+1)^2},
    \end{equation}
    in other words, replace the factor $(n+3)$ by $(n+1)$ in \eqref{eq:uex} and normalize such that $v_0=1$. Now the generating function $g(x) = \sum_{n\geq0} v_n x^n$ is given by 
    \[
        f(x) = \pFq{6}{5}{1/4, 1/2, 3/4, 3, 1, 1}{1/3,, 2/3,, 2,, 2,, 2}{\frac{256}{27}x} = 1+\frac{3}{2} x +\frac{56}{9} x^{2}+\frac{275}{8} x^{3} + \cdots.
    \]
    For the contraction we only remove one pair $(3,2)$ since it is the only pair $(a_j,b_k)$ with $a_j-b_k \in \N$. It follows that the contraction of $f(x)$ is not reduced and consequently neither $f^{\cc}(x)$ nor $f(x)$ is algebraic.
\end{ex}    

We also provide an example from combinatorics in which the algebraicity of a hypergeometric series was overlooked for several years.
    
\begin{ex} \label{ex:Gessel1}
    The generating function for the number of Gessel excursions (i.e., walks of length $n$ in the quarter plane starting and ending at the origin and having step set $\mathfrak{S} = \{ \nearrow, \swarrow, \leftarrow, \rightarrow \}$) was first conjectured by I.~Gessel in 2001 and then proven by Kauers, Koutschan and Zeilberger \cite{KaKoZe09} in 2009 to be given by
    \begin{equation} \label{eq:Gesseldef}
        G(x)=\pFq{3}{2}{5/6, 1/2, 1}{5/3, 2}{16x^2}=\bigF{5/6, 1/2}{5/3, 2}{16x^2}.
    \end{equation}
    The first proof of this fact heavily relied on computer computations and the authors actually conjectured \cite[\S4]{KaKoZe09} that no purely human proof is possible -- this claim was, however, disproven by Bostan, Kurkova and Raschel \cite{BoKuRa17} a few years later.

    Quite surprisingly, the algebraicity of the function $G(x)$ was overlooked until Bostan and Kauers~\cite{BoKa10} proved the algebraicity of the trivariate complete generating function $Q(x,y,t)$ of Gessel walks ending at the point $(i,j) \in \N^2$ (note that $G(x) = Q(0,0,x)$). This is astonishing since one way to see why $G(x)$ is algebraic is to observe that
    \begin{equation} \label{eq:gesseltrans}
        G(x) = \frac{1}{2x^2} \left( \pFq{2}{1}{-1/2,-1/6}{2/3}{16x^2}-1 \right),
    \end{equation}
    and for this $_{2}F_{1}$ Schwarz' classification applies (see \cref{ex:Gessel0}). Equation \eqref{eq:gesseltrans} is, however, somewhat ad-hoc; at the same time, the present work provides a direct way to conclude algebraicity of $G(x)$ from \eqref{eq:Gesseldef}. 
    
    Note that the direct application of the interlacing criterion by Beukers and Heckman~\cite{BH89} (see~\cref{thm:BeHe}) to $G(x)$ is not possible since in the ${}_{3}F_2$ representation some of the parameters have integral difference. Indeed, \eqref{eq:gesseltrans} immediately implies that the corresponding hypergeometric differential equation is reducible.
    
    On the other hand, the hypergeometric function in \eqref{eq:Gesseldef} is contracted and even reduced, so \cref{lem:BetterHe} applies and it therefore indeed holds that $G(x)$ is algebraic if and only if the corresponding hypergeonetric function satisfies the interlacing criterion (with parameters $(a_1,a_2) = (5/6,1/2)$ and $(b_1,b_2) = (5/3,2)$). Below, for $\lambda \in \{1,5\}$ the values of $\exp(2\pi i \lambda a_1), \exp(2\pi i \lambda a_2)$ are drawn in red and those of $\exp(2\pi i \lambda b_1),\exp(2\pi i \lambda b_2)$ are drawn in blue. One notices that the values indeed interlace on the unit circle and concludes that $G(x)$ is an algebraic function.
        \[
        \resizebox{0.4\textwidth}{!}{\BeHe{6}{5/6}{1/2}{1}{5/3}{2}}
        \]
    Applying our main result, \cref{thm:tree} (see also Figure~\ref{fig:tree}), directly to \eqref{eq:Gesseldef} also immediately reduces the algebraicity question to the interlacing criterion for the reduced hypergeometric function $\F([5/6,1/2],[5/3,2];x)$.
\end{ex}



\bibliographystyle{alphaabbr}
\bibliography{alg}

\textsc{Faculty of Mathematics, University of Vienna, Oskar-Morgenstern-Platz 1, 1090, Vienna, Austria}

\textit{Email: }\href{mailto:florian.fuernsinn@univie.ac.at}{\texttt{florian.fuernsinn@univie.ac.at}}\\

\textsc{Team CSAI, A\&R TECH GmbH, Dietz-von-Weidenberg-Gasse 2, 1210, Vienna, Austria}

\textsc{Faculty of Mathematics, University of Vienna, Oskar-Morgenstern-Platz 1, 1090, Vienna, Austria}

\textsc{Team MathExp, Inria Saclay Île-de-France, Bât. Alan Turing, 1 rue Honoré d'Estienne d'Orves, 91120, Palaiseau, France }

\textit{Email: }\href{mailto:sergey.yurkevich@univie.ac.at}{\texttt{sergey.yurkevich@univie.ac.at}}
\end{document}